\documentclass[bjps]{imsart}

\usepackage[utf8]{inputenc}
\usepackage[english]{babel}
\usepackage{amsmath,amssymb,amsthm,amsfonts}
\usepackage{graphicx}
\usepackage{bbm}
\usepackage{mathtools}
\usepackage{mathrsfs}
\usepackage{float}

\usepackage{subfig}
\usepackage[font=small,skip=0pt]{caption}
\usepackage[labelformat=brace]{subcaption}

\usepackage[skins,breakable]{tcolorbox}

\usepackage[authoryear]{natbib}

\usepackage{bookmark}
\bookmarksetup{depth=4}

\newcommand{\lfrac}[2]{\genfrac{}{}{0pt}{}{#1}{#2}}

\allowdisplaybreaks

\startlocaldefs

\theoremstyle{plain}

\newtheorem{theorem}{Theorem}[section]
\newtheorem{lemma}[theorem]{Lemma}

\newtheorem{example}{Example}
\newtheorem{Remark}{Remark}

\theoremstyle{remark}

\endlocaldefs

\usepackage[colorlinks,citecolor=blue,urlcolor=blue]{hyperref}

\begin{document}

\begin{frontmatter}
\title{Multiscaling limit theorems for stochastic FPDE with cyclic long-range dependence}
\runtitle{Multiscaling limit theorems for stochastic FPDE}

\begin{aug}

\author[A,B]{\inits{F.}\fnms{Maha Mosaad A}~\snm{Alghamdi}\ead[label=e0]{20312612@students.latrobe.edu.au}\ead[label=e1]{mmghamdi@iau.edu.sa}\orcid{0009-0004-6663-2286}},
  \author[C]{\inits{T.}\fnms{Nikolai}~\snm{Leonenko}\ead[label=e2]{LeonenkoN@cardiff.ac.uk}\orcid{0000-0003-1932-4091}},
\author[A]{\inits{S.}\fnms{Andriy}~\snm{Olenko}\ead[label=e3]{a.olenko@latrobe.edu.au}\orcid{0000-0002-0917-7000}}

\address[A]{Department of Mathematical and Physical Sciences,
  La Trobe University, Melbourne, Australia
  \printead[presep={,\ }]{e0,e3}}

\address[B]{Department of Mathematics, College of Science and Humanities,
  Imam Abdulrahman Bin Faisal University, Jubail 31441, Saudi Arabia
  \printead[presep={,\ }]{e1}}

\address[C]{School of Mathematics,
  Cardiff University, Senghennydd Road, Cardiff CF24 4YH, UK
  \printead[presep={,\ }]{e2}}
\end{aug}

\begin{abstract}
The paper studies solutions of stochastic partial
differential equations with random initial conditions. First, it overviews some of the known results on scaled solutions of such equations and provides several explicit motivating examples. Then, it proves multiscaling limit theorems for renormalized solutions for the case of initial conditions subordinated to random processes with cyclic long-range dependence. Two cases of stochastic partial differential equations are examined. The spectral and covariance representations for the corresponding limit random fields are derived. Additionally, it is discussed why analogous results are not valid for subordinated cases with Hermite ranks greater than~1.  Numerical examples that illustrate the obtained theoretical results are presented.
\end{abstract}

\end{frontmatter}

\section{Introduction}
The heat equation with random initial conditions is a well-studied topic in both mathematical and physical literature.
 \cite{de1956random} and \cite{rosenblatt1968remarks} introduced rigorous probabilistic methods to the heat equation with stationary initial conditions and gave the spectrum representations of stationary solutions as stochastic integrals.
Heat equation solutions have been studied under various random initial conditions, including non-homogeneous cases with random potentials, see, for example, \cite{becus1980variational, uboe1995stability}.

 Random fields, which are often obtained as solutions of such stochastic partial differential equations (SPDE), are widely used for modelling in numerous applications, for instance,  earth sciences (\cite{CG, CG1}), geophysics (\cite{Fisher}),
    climatology (\cite{Oh}) and cosmology,  see
    \cite{BKLO, BKLO1, Cabella,   M19, MP} and the references therein. Applications of stochastic methods in
   such areas have become increasingly important due to the enormous experimental data obtained in recent years, see~\cite{Adam}. Studying the evolution and properties of such random fields is of particular interest for the modelling and analysis of Cosmic Microwave Background radiation.  Several models were recently introduced in \cite{BKLO, BKLO1, LNO}, where SPDEs were used to describe changes in those random fields over time.

 Random fields with singular spectra appear in various problems, including rescaling linear diffusion equations with singular initial conditions, see \cite{ Anh1999NonGaussianSF, ANH2000239, Anh2002RenormalizationAH, Leonenko1998ScalingLO}. Several researchers have studied the Burgers equation with random data and the Cole-Hopf transformation, relating it to the heat equation, see \cite{Leonenko1998ExactPA, Leonenko1999LimitTF, Woyczynski1998BurgersKPZT}. \cite{beghin2000gaussian} studied scaling laws for linear Korteweg-de Vries or Airy equations using random data. \cite{ANH2000239, Anh2002RenormalizationAH} proposed the theory for the renormalization and homogenization of fractional-in-time or in-space diffusion equations using random data. See also \cite{10.1214/EJP.v16-896} and the references therein.

\cite{Anh1999NonGaussianSF, ANH2000239, Anh2002RenormalizationAH} and  \cite{ruiz2001scaling} investigated fractional-in-time and fractional-in-space diffusion equations. They propose new Gaussian and non-Gaussian solutions for the renormalized fractional diffusion equation with random data. Their results are similar to non-Gaussian central limit theorems for solving generalized kinetic equations with singular data, as seen in \cite{dobrushin1979non, taqqu1979convergence, anh2001spectral, Anh2002RenormalizationAH} and references therein.

This paper generalises the mentioned results in several directions, especially in the case when the initial condition random process is subordinated to the random process with cyclic long-range dependence. As was discussed in \cite{Ivanov2013, Olenko2013}, cyclic long-range dependence with spectral singularities at non-zero frequencies leads to non-standard asymptotics. The limit theorems for such scenarios are derived for rescaled solutions and the covariance structures of the limit processes are given. The principal technical challenge in the proofs arises from the presence of singularities at non-zero frequencies, which precludes the use of the standard approach. In contrast to the case of a singularity at zero frequency, which remains unchanged under rescaling, the singularity locations in the considered scenario exhibit a different behaviour. As the rescaling factor decreases to zero, the singularity locations shift towards infinity. Therefore, an integrable upper bound does not exist, making it impossible to apply the standard asymptotic methods based on the classical dominated convergence theorems. This necessitates careful modification of the standard asymptotic approach. It is interesting that in contrast to the stationary limits discussed in the previous publications, the limit field exhibits stationarity in the spatial domain, but lacks stationarity in the temporal domain. Another evidence indicating that the case under consideration is more difficult and complex than the established scenarios is that, for subordinate cases with Hermite ranks greater than 1, a limit does not exist, contrary to known results in the literature. Specifically, there is no appropriate normalization for the scaled solutions to achieve non-degenerate limit behaviour, see the corresponding discussion in Section \ref{sec6}. In this case, our covariance function is analogous to the so-called Berry’s random wave model in \cite{maini2024spectral}, and we expect that the first non-zero Hermite coefficient is not necessarily playing the main role in the asymptotic theory. We plan to investigate this problem in the future.

The paper has the following structure. Section~\ref{sec2} gives the main definitions and assumptions. It introduces the considered model and demonstrates it through various examples. For the comparison, two known results for long-range dependent fields are presented. Section~\ref{sec3} and~\ref{sec4} prove the main results about multiscaling limit theorems under cyclic long-range dependent scenarios.
The paper also presents numerical examples that illustrate the obtained results. Some future research directions are discussed in Section~\ref{sec6}.

All numerical computations, simulations and plotting in this paper were performed using the software R (version 4.4.1) and Maple~2023. The corresponding R
and Maple code is freely available in the folder
”Research materials” from the website\\ \url{https://sites.google.com/site/olenkoandriy/}.

\section{Main definitions, models, assumptions and examples}
\label{sec2}

Let $(\Omega, \mathcal{F}, P)$ be a complete probability space. This paper studies spatio-temporal random fields $u(t,x,\omega), $ $ t\in \mathbb R_+:=\{t \in \mathbb R:t\ge 0\},$ $x\in {\mathbb R^d},$ $d\in \mathbb{N},$ and $\omega\in \Omega.$  For simplicity, the argument $\omega$ will be suppressed, unless its exclusion results in confusion.

For an appropriate class of functions the Caputo-Djrbashian non-local operators or fractional derivative of order $\beta \in (0,1]$ is defined as follows (for more details see \cite[p.4]{almeida2017caputo}, \cite[remark 2.1 (ii)]{d2021non}, \cite[p.39]{MeerschaertSikorskii}, \cite[(2.138)]{podlubny1998fractional}, \cite[p.4]{tavares2016caputo} )
\begin{align*}
\partial^\beta _t u(t,x):= \frac{\partial^\beta}{\partial t^\beta } u(t,x) &: = \frac{1}{\Gamma(1-\beta)}\left(\frac{\partial}{\partial t}\int^t_0\frac{u(\tau,x)d\tau}{(t-\tau)^\beta}-\frac{u(0,x)}{t^\beta}\right)\\
&= \frac{1}{\Gamma(1-\beta)}\int^t_0\frac{\frac{\partial}{\partial\tau}u(\tau,x)}{(t-\tau)^{\beta}}d\tau, \quad\beta\in(0,1),
\end{align*}
and
\[\frac{\partial^{\beta}}{\partial t^{\beta}}u(t,x)=\frac{\partial u}{\partial t}(t,x), \quad\text{if}\quad \beta=1,\]
where $\Gamma(\cdot)$ denotes the Gamma function.

Consider the following operator
\[\partial ^{\beta}_t+\mu(I-\Delta)^{\gamma/2}(-\Delta)^{\alpha/2}, \quad\alpha \geq 0, \quad \gamma>0, \quad \mu >0,\]
where $\Delta$ is the $d$-dimensional Laplace operator. The operators $-(I-\Delta)^{\gamma/2}$  and $(-\Delta)^{\alpha/2}$  are interpreted as inverses of the Bessel and Riesz potentials respectively, see, for example, \cite{Anh1999NonGaussianSF, ANH2000239}. For $\gamma\in \mathbb{R_+},$ the integral operator \[\mathfrak{I}_{\gamma}=(I-\Delta)^{-\gamma/2}\]  is called the Bessel potential of order $\gamma$, where the kernel $\mathfrak{I}_{\gamma}$ is given by \[\mathfrak {I}_{\gamma}(x)=\frac{1}{(4\pi)^{\gamma/2}\Gamma(\gamma/2)}\int_{0}^{\infty}e^{-\pi|x|^{2}/s}e^{-s/{4\pi}}s^{(-n+\gamma)/2}\frac{ds}{s}.\]  The Riesz potential is defined by $\mathfrak{J}_{\alpha}=(-\Delta)^{-\alpha/2},$ $0<\alpha<n.$ Then, for $f\in S(\mathbb{R}^{n})$,
\[  \mathfrak{J_{\alpha}}(f)(x)=\frac{1}{g(\alpha)}\int_{\mathbb{R}^{n}}|x-y|^{\alpha-n}f(y)dy=(J_{\alpha} \ast f)(x),
\]
where\[g(\alpha)=\frac{\pi^{n/2}2^{\alpha}\Gamma(\alpha/2)}{\Gamma(n/2-\alpha/2)},\] and $J_{\alpha}(t)={|t|^{\alpha-n}}/{g(\alpha)}$ is the Riesz kernel, see Appendix B in \cite{anh2001spectral}.

The Green function $G(t,x)$ of this operator is defined via its spatial Fourier transform
\[\widehat{G}(t,\lambda)=\int_{\mathbb R^d}G(t,x)e^{i\langle x,\lambda\rangle}dx=E_{\beta}\left(-\mu\Vert {\lambda}\Vert^\alpha(1+\Vert{\lambda}\Vert^{2})^{\gamma/2}t^{\beta}\right), \quad\lambda \in \mathbb{R}^d,\]
where $E_{\beta}(\cdot)$ is the Mittag-Leffler function given  by the series
\[E_{\beta}(z):=\sum_{k=0}^{\infty}\frac{z^{k}}{\Gamma(\beta{k}+1)}, \quad z\in \mathbb {C}, \quad 0<\beta<1.\]
Note that if $\beta=1$, then
$E_{1}(-z)=e^{-z},$ $\quad z\geq 0.
$

For real $u \geq 0$ it satisfies the inequality
\begin{equation}\label{mattig cond}
  \frac{1}{1+\Gamma(1-\beta)u} \leq E_{\beta}(-u)\leq \frac{1}{1+\frac{u}{\Gamma(1+\beta)}}, \quad u\geq 0.
\end{equation}
For more details see \cite{mainardi2014some, Simon2013ComparingFA}.

This paper investigates the following SPDE.

\paragraph{Model:}  The non-local fractional Riesz-Bessel equation~(FRBE) is given by
\begin{equation}\label{Riesz-Bessel equation}
   \frac{\partial^{\beta}}{\partial t^{\beta}}u(t,x)=-\mu(I-\Delta)^{\gamma/2}(-\Delta)^{\alpha/2}u(t,x) , \quad t>0,\quad x\in \mathbb R^d,
\end{equation}
 subject to the random initial condition
 \begin{equation}\label{initial condition}
      u(0,x,\omega)=\eta(x,\omega), \quad x \in \mathbb R^d, \quad \omega\in \Omega,
 \end{equation}
      where $\eta(x,\omega)$ is a measurable random field on the probability space $(\Omega, \mathcal{F}, P).$
      The solution $u(t,x)$ can be interpreted as a mean-square solution of the initial value problem ~(\ref{Riesz-Bessel equation})-(\ref{initial condition}), see (\cite{ANH200377}).

      For simplicity, the notations $u(t,x)$ and $\eta(x)$ with the suppressed argument $\omega$ will be used in the future.

Note that
\[\int_{\mathbb R^d} G(t,x)dx=1, \quad \text{for all}\quad t\geq 0,\] and the Green function solution of (\ref{Riesz-Bessel equation}) has the form
\[
  u(t,x)=\int_{\mathbb R^d} G(t,x-y)\eta(y)dy.
\]

To introduce the class of random fields used as the initial condition (\ref{initial condition}) for solving the equation (\ref{Riesz-Bessel equation}), the following assumptions are required.

\paragraph{Condition A:}The function $h(\cdot)\in L_2\left(\mathbb R, \varphi(u)du\right),$ i.e. such that \[\int_{\mathbb R}h{^2}(u)\varphi(u)du <\infty,\]
where
$\varphi(u)=e^{-u^2/2}/\sqrt{2\pi},$ $u\in\mathbb R,$
is the standard Gaussian density.

Such functions $h(\cdot)$ can be expanded in the series
\[h(u)=C_0+\sum_{k=1}^{\infty}\frac{C_k}{k!}H_{k}(u),\quad C_k:=\int_{\mathbb R}h(u)\varphi(u)H_{k}(u)du, \]
where $\{H_{k}(u)\}_{k=0}^{\infty}$ are the Hermite polynomials, see \cite{Peccati2011}.
\paragraph{Condition B:} There exists an integer $m\geq 1$ such that
$C_{1}=...=C_{m-1}=0,$ $ C_{m}\neq 0.$
The integer $m$ is called the Hermitian rank of the function $h(\cdot).$

\paragraph{Condition C:} The random initial condition field is $\eta(x)=h(\xi(x)),$ $ x\in \mathbb R^d,$
where $\xi(x),$ $ x\in \mathbb R^{d},$ is a real-valued, measurable, mean square continuous, homogeneous and isotropic Gaussian random field with $\mathbb E\xi(x)=0$ and $\mathbb E\xi^2(x)=1.$ The spectral measure of random field $\xi(x)$  is absolutely continuous in the sense given below.

Its covariance function $B_\xi(\Vert x\Vert):=Cov(\xi(0),\xi(x))$ depends only on $\Vert x\Vert$ and can be given as
\begin{equation}\label{cov}
B_\xi(\Vert x\Vert) = \int_{\mathbb R^d}e^{i\langle\lambda,x\rangle}dF_\xi(\lambda)  =  \int_{0}^{\infty}Y_d(r\Vert x\Vert)\Phi(dr),
\end{equation}
where $F_\xi(\cdot)$ is the spectral measure of the field $\xi(x)$ on $(\mathbb R^{d}, \mathscr{B}(\mathbb R^{d}))$, $\Phi(\cdot)$ is the corresponding isotropic spectral measure, the function $Y_d(\cdot)$ is defined by
\[Y_d(r):=\frac{ 2^\nu\Gamma(\nu+1)}{r^\nu}J_{\nu}(r),\quad \nu=\frac{d-2}{2}, \quad  r\ge 0,\]
and $J_\nu(\cdot)$ is the Bessel function of the first kind of order $\nu.$

The absolutely continuous spectral measure $F_{\xi}(\cdot)$  can be represented as
\[F(\Delta)=\int_{\Delta} f_{\xi}(\lambda)d\lambda,\quad \Delta\in \mathscr{B}(\mathbb R^{d}). \]
The function $f_{\xi}(\lambda), \lambda\in \mathbb R^{d},$ which is integrable over $\mathbb R^{d}$, is called the spectral
density function of the homogeneous random field $\xi(\cdot)$.
In this case the spectral representation
(\ref{cov}) can be written as \[B_\xi(\Vert x\Vert) = \int_{\mathbb R^d}e^{i\langle\lambda,x\rangle}f_{\xi}(\lambda)d\lambda,\]
see {\cite{Statisticalanalysis, Leonenko1999LimitTF}}.
\begin{example}
For $d=1,2,$ and $3$ it holds  $Y_1(r)=\cos(r),$ $Y_2(r)=J_0(r),$ and $Y_{3}(r)={\rm sinc}(r)={\sin(r)}/{r}$ respectively.
\end{example}

It follows from Conditions A and C that $\mathbb Eh^2(\xi(x))<\infty$ and the random field~$\xi(x)$ has the following isonormal spectral representation
\begin{equation}\label{spec_xi}\xi(x)=\int_{\mathbb R^d}e^{i\langle \lambda,x \rangle} \sqrt{f_{\xi}(\lambda)}W(d\lambda),
\end{equation}
 where $W(\cdot)$ is the white-noise random measure on $\mathbb{R}^d$, for more details see \cite[Ch.~2, p.~114]{Statisticalanalysis}, and $f_{\xi}(\lambda)\in L_{1}(\mathbb R^{d})$ is spectral density of the field $\xi(\cdot).$

Using the following spectral representation of the random field $\eta(x)$
\[\eta(x)=\int_{\mathbb R^{d}}e^{i\langle\lambda,x\rangle}Z(d\lambda),\quad \mathbb E|Z(d\lambda)|^{2}=F_{h}(d\lambda),\]
where \[Cov(\eta(0),\eta(x))=\int_{\mathbb R^{d}}e^{i\langle\lambda,x\rangle}F_{h}(d\lambda), \]
one obtains the following spectral representation of the solution of the initial value problem (\ref{Riesz-Bessel equation}) and (\ref{initial condition})
 \begin{align}{\label{spectral 1}}
 u(t,x)&=\int_{\mathbb R^{d}} G(t,x-y)\eta(y)dy \nonumber \\&=\int_{\mathbb R^{d}}e^{i\langle\lambda,x\rangle}E_{\beta}\left(-\mu\Vert {\lambda} \Vert^\alpha(1+\Vert{\lambda}\Vert^{2})^{\gamma/2}t^{\beta}\right)Z(d\lambda).
 \end{align}

 Then, the covariance function of the random field (\ref{spectral 1}) has the form
 \begin{align*}
      Cov\left(u(t,x),u(t',x')\right)=&\int_{\mathbb R^{d}}e^{i\langle\lambda,x-x'\rangle}E_{\beta}\left(-\mu\Vert{\lambda} \Vert^\alpha(1+\Vert{\lambda}\Vert^{2})^{\gamma/2}t^{\beta}\right) &\\&\times E_{\beta}\left(-\mu\Vert {\lambda} \Vert^\alpha(1+\Vert{\lambda}\Vert^{2})^{\gamma/2}(t')^{\beta}\right)F_h(d\lambda).
 \end{align*}

 The following examples present several important well-known particular cases of the general model.

 \begin{example} If $\beta=1,$ $ \gamma=0$, $\alpha=2,$
 then equation {\rm (\ref{Riesz-Bessel equation})} becomes the heat equation, i.e.
 \[  \frac{\partial u(t,x)}{\partial t}=\mu\Delta u(t,x), \quad t>0,\quad x\in \mathbb R^d,
\] and its solution has the following spectral representation
 \[u(t,x)=\int_{\mathbb R^{d}}e^{i\langle\lambda,x\rangle-\mu t\Vert \lambda\Vert^{2}} Z_{h}(d\lambda).\]
  \end{example}
 \begin{example} For the case of $\beta \in(0,1],$ $ \gamma=0,$ and $ \alpha=2,$ the equation {\rm (\ref{Riesz-Bessel equation})} and the integral in {\rm (\ref{spectral 1})} reduce to
  \[   \frac{\partial^{\beta}}{\partial t^{\beta}}u(t,x)=\mu\Delta u(t,x), \quad t>0,\quad x\in \mathbb R^d,
\]
 \[u(t,x)=\int_{\mathbb R^{d}}e^{i\langle\lambda,x\rangle}E_{\beta}\left(-\mu t^{\beta}\Vert \lambda\Vert^{2}\right)Z_{h}(d\lambda).\]
 \end{example}

To establish limit theorems for the solutions of equation (\ref{Riesz-Bessel equation}), one needs to incorporate additional assumptions on the dependence structure of the random field~$\xi(x)$ governing the initial condition (\ref{initial condition}).

\paragraph{Condition D:} The random field $\xi(x)$ in Condition C has the covariance function of the form
\[{B}_{\xi}(\Vert x\Vert)=\frac{L(\Vert x\Vert)}{\Vert x \Vert^{\kappa}}, \quad 0<\kappa<\frac{d}{m},\]
where $L:(0,\infty)\to (0,\infty),$ is a slowly varying at infinity function, see \cite{Bingham_Goldie_Teugels}.

\begin{Remark}\label{remark 1}
Note that the covariance functions satisfying Condition~D are non-integrable, implying that the random field $\xi(x)$ exhibits long-range dependence.
\end{Remark}
\begin{Remark}
By Tauberian-Abelian theorems, under some minor assumptions on the spectral density, see {\rm \cite{Leonenko1991TauberianAA, Leonenko_Olenko_2013}}, Condition D is equivalent to the following spectral conditions. These types of conditions are often more convenient for proofs.
\end{Remark}

\paragraph{Condition \texorpdfstring{D${}^\prime$}{D'}:}
The random field $\xi(x)$ in Condition C has the spectral density of the form
\[f_{\xi}(\Vert \lambda \Vert)= C(d,\kappa)\frac{L_0({1}/{\Vert \lambda \Vert})}{\Vert \lambda \Vert^{d-\kappa}}, \quad \lambda \in \mathbb{R}^d,\] where $L_0(\cdot)$ is a slowly varying functions such that $L_0(r) \sim L(r),$ when $r\to+\infty,$ and \[C(d,\kappa):=\frac{\Gamma\left(\frac{d-\kappa}{2}\right)}{2^{\kappa}\pi^{d/2}\Gamma\left(\frac{\kappa}{2}\right)}.\]

Let $W(\cdot)$ be the white-noise random measure from (\ref{spec_xi}),
$\int_{\mathbb{R}^{d_m}}'(\ldots)W(d\lambda_{1})\ldots W(d\lambda_{m})$
denote the multiple Wiener-It\^{o} stochastic integral with respect to $W$, with the diagonals $\lambda_k=\pm\lambda_j$,
$k, j =1,\ldots,m$, $k\neq j$, being excluded from the domain of integration.

For the introduced classes of random fields, the following asymptotic results provide multiscaling limit theorems for solutions of the initial value problem (\ref{Riesz-Bessel equation}) and (\ref{initial condition}) for the classical long-range dependent case.
\begin{theorem}{\rm \cite{anh2001spectral}} \label{th1} Consider the random field $u(t,x),$ $ t\ge 0,$ $ x\in \mathbb R^d,$ defined by {\rm (\ref{spectral 1})} in which $\eta(x),$ $ x\in \mathbb R^d,$ satisfies Condition~C with the rank $m\geq 1$ and $\xi(x),$ $ x\in \mathbb R^d,$ satisfies condition A, B and D${}^\prime$ with $\kappa<\frac{min(2\alpha,d)}{m}.$

Then, when $\varepsilon\to 0,$ the finite-dimensional distributions of the random fields
\[U_{\varepsilon}(t,x):={\varepsilon^{-\frac{m\kappa\beta}{2\alpha}}}u\left(\frac{t}{\varepsilon},\frac{x}{\varepsilon^{\beta/\alpha}}\right), \quad \varepsilon>0, \quad t>0, \quad x\in \mathbb R^d,\]
converges  to the finite-dimensional distributions of the field
\begin{align*}
    U_{m}(t,x):=&\frac{C_m}{m!}C^{m/2}(d,x)\int_{\mathbb R^{d_m}}'\frac{e^{i\langle x,\lambda_{1}+...+\lambda_{m} \rangle}}{\left(\|\lambda_{1}\|...\|\lambda_{m}\|\right)^{\frac{d-\kappa}{2}}} \\
   &\times E_{\beta}\left(-\mu t^{\beta}\|\lambda_{1}+...+\lambda_{m}\|^{\alpha}\right)W(d\lambda_{1})...W(d\lambda_{m}).
\end{align*}
    \end{theorem}
\begin{Remark} The assumption $\kappa < \min(2\alpha, d)/m$ guaranties that the random field $\xi(x)$ is long-range dependent, see Remark \ref{remark 1}. At the same time, as the spectral density of $U_1(t,x)$ has the next asymptotic behavior, see equation (4.24) in \cite{anh2001spectral},
$$\frac{c}{\mu t^{\beta}||\lambda||^{d+2\alpha-\kappa}}+O\left(\frac{1}{||\lambda||^{d+2\alpha-\kappa+1}}\right),$$
the assumption guarantees the integrability of the spectral density. The multipliers $||\lambda||^{d + 2\alpha - \kappa}$ and $t^{\beta}$ in the leading term indicate the presence of second-order intermittency and non-exponential relaxation function, consult \cite{anh2001spectral}.
    \end{Remark}
\begin{theorem}{\rm \cite{Anh2002RenormalizationAH}}\label{th2}
Consider the random field $u(t,x),$ $ t\ge 0,$ $ x\in \mathbb R^d,$ defined by {\rm (\ref{spectral 1})} in which $\eta(x),$ $ x\in \mathbb R^d,$ satisfy Condition~C with the rank $m\geq 1$ and $\xi(x),$ $ x\in \mathbb R^d,$ satisfy Condition~D${}^\prime$,  some additional conditions on one and two-dimensional densities given in {\rm\cite[$A''$]{Anh2002RenormalizationAH}} and
 \[\int_{\mathbb R^{d}}|B_{\xi}(x)|^{m}dx<\infty, \quad K:=\sum_{k=m}^{\infty}\frac{C_{k}^{2}}{k!}\int_{\mathbb R^{d}}B_{\xi}^{k}(x)dx>0.\] Then, the finite-dimensional distributions of the random fields
 \[\tilde{U}_{\varepsilon}(t,x):={\varepsilon^{-d\beta/2\alpha}}u\left(\frac{t}{\varepsilon},\frac{x}{\varepsilon^{\beta/\alpha}}\right), \quad \varepsilon>0, \quad t>0, \quad x\in \mathbb R^d,\]
 converge as $\varepsilon\to 0$ to the finite-dimensional distribution of the homogenous Gaussian field $\tilde{U}(t,x),$ $ t> 0,$ $ x\in \mathbb R^d,$
 with zero mean and the covariance function
 \[E\tilde{U}(t,x)\tilde{U}(t',x')=KG(t-t',||x-x'||), \]
where $G(\cdot)$ is the corresponding Green function.
\end{theorem}
This paper aims to generalise these results and investigate multiscaling behaviour for the case of $\xi(x)$ that exhibits cyclic long-range dependence.

\section{Multiscaling limits for the heat equation with cyclic long-range dependent initial conditions}\label{sec3}
This section gives analogous results to Theorems~\ref{th1} and~\ref{th2} in the case of the heat equation when the initial stochastic condition has cyclic long-range dependence. For simplicity, we concentrate on the one-dimensional case with~$d=1.$

 We consider the classical heat equation
\begin{equation}\label{classical heat equation}
   \frac{\partial u(t,x)}{\partial t}=\mu\frac{\partial^2 u(t,x)}{\partial x^2},\quad t>0,\quad x\in \mathbb R,
\end{equation}
subject to random initial
condition
\begin{equation}\label{initial condition for heat equation}
   u(0,x)=\eta(x)=\xi(x)=\int_{\mathbb  R }e^{i\lambda x}Z(d\lambda),\quad x\in\ \mathbb R,
\end{equation}
where $\mathbb{E}|Z(d\lambda)|^{2}=F(d\lambda).$

Note, that in this case $h(x)\equiv x$ in Conditions~A and C and $m=1$ in Condition~B.

The process $\eta (x),$ $ x\in \mathbb R,$ is a stationary Gaussian zero-mean unit-variance stochastic process with the covariance function $r(x):=Cov(\eta(0),\eta(x)),$ $x\in \mathbb R.$
The Bochner-Khinchin theorem assures that the covariance function $r(x)$ has the spectral representation
\begin{equation}\label{covariance function of heat equation}
    r(x) =\int_{\mathbb R}e^{i\lambda x}F(d\lambda), \quad r(0)=1,
\end{equation}
  where the spectral measure $F$ is bounded and positive on $(\mathbb R, \mathscr{B}(\mathbb R)).$

The solution $u(t,x),$ $ t>0,$ $ x\in \mathbb R,$ of the initial value problem  (\ref{classical heat equation})- (\ref{initial condition for heat equation}) can be written as the convolution
\[ u(t,x)=\frac{1}{\sqrt{4\pi\mu t}}\int_{\mathbb R} \eta(y){e^{-\frac{(x-y)^2}{4\mu t}}}dy.\]
Therefore, the solution field $u(t,x)$ is a stationary process in $x$ with the spectral representation
\begin{equation}\label{convolution}
    u(t,x)=\int_{\mathbb R}e^{i\lambda x-\mu \lambda^{2}t}Z(d\lambda),
\end{equation}
has a zero mean and the covariance function
\[
    Cov(u(t,x),u(t',x'))=\int_{\mathbb R}e^{i \lambda (x-x')-\mu\lambda^{2}(t+t')}F(d\lambda).
\]
It shows that the solution is not stationary with respect to time.

We generalise Condition D as follows.

\textbf{Condition D${}^{\prime\prime}$:}  The covariance function (\ref{covariance function of heat equation}) is of the form
\begin{equation}\label{covariance of our problem}
    r_{\kappa,w}(x):=r(x)=\frac{\cos(wx)}{(1+x^{2})^{\kappa/2}},\quad x\in\mathbb R,\quad w\neq 0,\quad 0<\kappa<1.
\end{equation}
This covariance function is non-integrable and has an oscillating behaviour which corresponds to the cyclic long-range dependence scenario.
It follows from (\ref{covariance of our problem}) and~{\cite{Ivanov2013}} that the spectral measure has the representation  $F(d\lambda)=f_{\kappa,w}(\lambda)d\lambda,$ $ \lambda \in \mathbb R,$ with
\begin{equation}\label{spectral0}
f_{\kappa,w}(\lambda):=\frac{c_1(\kappa)}{2}\left(K_{\frac{\kappa-1}{2}}\left(|\lambda+w| \right)|\lambda+w|^{\frac{\kappa-1}{2}}+ K_{\frac{\kappa-1}{2}}\left(|\lambda-w| \right)|\lambda-w|^{\frac{\kappa-1}{2}} \right),
\end{equation}
where $c_1(\kappa):={2^{\frac{1-\kappa}{2}}}/\left({\sqrt{\pi}{\Gamma\left(\frac{\kappa}{2} \right)}}\right)$  and $K_\nu(\cdot)$ is the modified Bessel function of the second kind. The function $K_\nu(z)$ is given as follows \cite{Ivanov2013} \[K_\nu(z)=\frac{1}{2}\int_{0}^{\infty}s^{\nu-1}\exp\left(-\frac{1}{2}\left(s+\frac{1}{2} \right)z\right)ds,\quad z\geq 0,\quad \nu\in \mathbb {R}.\]

\begin{Remark}
As the spectral density $f_{\kappa,\omega}(\cdot)$ is an even function, all random processes considered hereafter are real-valued.
\end{Remark}

Figure~{\ref{fig1}} shows the covariance function~({\ref{covariance of our problem}}) and spectral density~({\ref{spectral0}}) for the case of $\kappa=0.8$ and $w=1.$
\begin{figure}[tb]
    \centering
    \subfloat{{\includegraphics[scale=0.37]{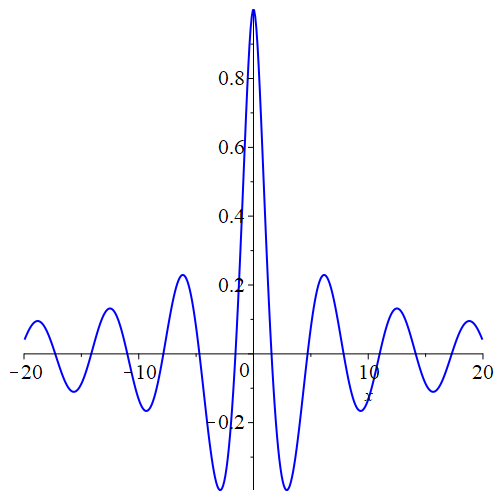} }}
    \qquad
    \subfloat{{\includegraphics[scale=0.37]{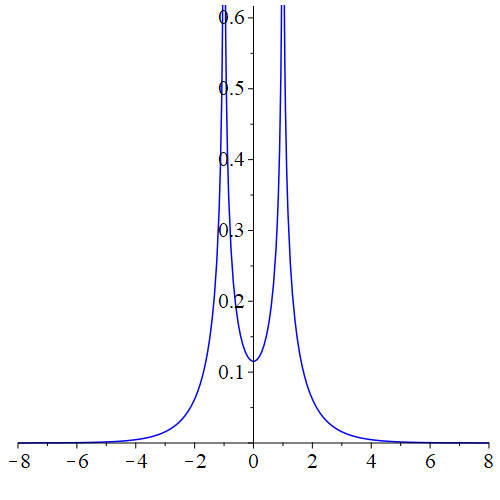} }}
    \\
    \caption{Covariance function and spectral density for $w=1$ and $\kappa=0.8.$}
    \label{fig1}
\end{figure}
\begin{lemma}\label{lem1} The spectral density $f_{\kappa,w}(\cdot)$ has the following representation
\begin{equation}\label{spectral}
f_{\kappa,w}(\lambda)=\frac{c_{2}(\kappa)}{2}\left(\frac{1-\theta_\kappa \left(|\lambda+w|\right)}{|\lambda+w|^{1-\kappa}} +\frac{1-\theta_\kappa \left(|\lambda-w| \right)}{|\lambda-w|^{1-\kappa}}\right),
\end{equation}
where $c_{2}(\kappa):=\left(2\Gamma(\kappa)\cos\left(\frac{\kappa\pi}{2}\right)\right)^{-1}.$

For each $\kappa\in(0,1)$ the function $\theta_\kappa(u)$ is bounded and $\theta_\kappa(u)\leq 1$ on $u\in[0,\infty),$ and it holds
\begin{equation}\label{condition1}
\theta_\kappa(u)=\frac{\Gamma\left(\frac{\kappa+1}{2} \right)}{2^{1-\kappa}\Gamma\left(\frac{3-\kappa}{2} \right)  }\left|{u}\right|^{1-\kappa}- \frac{|u|^{2}}{2(\kappa+1) }+o(|u|^{2}),\quad u\to 0,
\end{equation}
\begin{equation}\label{condition2}
 f_{\kappa,0}(\lambda)\sim\frac{c_{1}(\kappa)\sqrt{\pi}}{\sqrt{2}} |\lambda|^{\frac{\kappa-2}{2}} e^{-|\lambda|},\quad |\lambda| \to \infty.
\end{equation}
\end{lemma}

\begin{proof} By {\rm\cite[10.27.4]{NIST:DLMF}}, for non-integer values $\nu$ the modified Bessel function of the second kind can be represented as
\[K_{\nu}\left(z\right)=\frac{\pi}{2\sin\left(\nu \pi\right)}\left(I_{-\nu}\left(z\right)- I_{\nu}\left(z\right)\right),\]
and the limit is used if $\nu$ is an integer. So, the formula above is correctly defined, where $I_{\nu}(\cdot)$ is a modified Bessel function of the first kind, which is specified as follows~{\rm\cite[10.25.2]{NIST:DLMF}}
\[I_{
\nu}(z)=\left(\frac{z}{2}\right)^\nu\sum_{m=0}^\infty\frac{(z/2)^{2m}}{m!\,\Gamma(m+\nu+1)}.\]
Hence,
\begin{align}
  & K_{\nu}\left(|u|\right)|u|^{\nu}=\frac{\pi|u|^{\nu}}{2\sin\left(\nu \pi\right)}\left(I_{-\nu}\left(|u|\right)- I_{\nu}\left(|u|\right)\right) \nonumber \\ &\quad =\frac{\pi|u|^{\nu}}{2\sin\left(\nu \pi\right)}\left(\left(\frac{|u|}{2}\right)^{-\nu}\sum_{m=0}^\infty\frac{(|u|/2)^{2m}}{m!\,\Gamma(m-\nu+1)}-\left(\frac{|u|}{2}\right)^\nu\sum_{m=0}^\infty\frac{(|u|/2)^{2m}}{m!\,\Gamma(m+\nu+1)}\right) \nonumber\\
&\quad =\frac{\pi}{2\sin\left(\nu \pi\right)}\left(\frac{2^\nu}{\Gamma(1-\nu)}- \frac{|u|^{2\nu}}{2^\nu\Gamma(\nu+1)}+2^{\nu}\sum_{m=1}^\infty\frac{(|u|/2)^{2m}}{m!\,\Gamma(m-\nu+1)}-\left(\frac{|u|^2}{2}\right)^\nu\right.\nonumber\\
&\quad \left.\times\sum_{m=1}^\infty\frac{(|u|/2)^{2m}}{m!\,\Gamma(m+\nu+1)}\right)=
-\frac{\pi}{2^{\nu+1}\Gamma(\nu+1)\sin\left(\nu \pi\right)} |u|^{2\nu}\left(1-\Bigg(\frac{2^{2\nu}|u|^{-2\nu}}{\Gamma(1-\nu)}\right.\nonumber\\
&\quad \left. - \sum_{m=1}^\infty\frac{(|u|/2)^{2m-2\nu}}{m!\,\Gamma(m-\nu+1)} + \sum_{m=1}^\infty\frac{(|u|/2)^{2m}}{m!\,\Gamma(m+\nu+1)}\right)\Gamma(\nu+1)\Bigg).\nonumber
\end{align}
Thus, by (\ref{spectral0}),
\begin{align}
&f_{\kappa,w}(\lambda)=\frac{c_1(\kappa)}{2}\left(K_{\frac{\kappa-1}{2}}\left(|\lambda+w| \right)|\lambda+w|^{\frac{\kappa-1}{2}}+ K_{\frac{\kappa-1}{2}}\left(|\lambda-w| \right)|\lambda-w|^{\frac{\kappa-1}{2}} \right)\nonumber\\
&\quad =-\frac{\pi c_1(\kappa)}{2^{\frac{\kappa+1}{2}+1}\Gamma(\frac{\kappa+1}{2})\sin\left(\frac{\kappa-1}{2} \pi\right)}
\left(\frac{1-\theta_\kappa \left(|\lambda+w|\right)}{|\lambda+w|^{1-\kappa}} +\frac{1-\theta_\kappa \left(|\lambda-w| \right)}{|\lambda-w|^{1-\kappa}}\right),\nonumber
\end{align}
where
\begin{align}
\theta_\kappa (|u|) &= \Gamma\left(\frac{\kappa+1}{2}\right)\Bigg(\frac{2^{\kappa-1}}{\Gamma\left(1-{\frac{\kappa-1}{2}}\right)|u|^{\kappa-1}} + \sum_{m=1}^\infty\frac{(|u|/2)^{2m-\kappa+1}}{m!\,\Gamma(m-{\frac{\kappa-1}{2}}+1)} \nonumber\\
& -\sum_{m=1}^\infty\frac{(|u|/2)^{2m}}{m!\,\Gamma(m+{\frac{\kappa-1}{2}}+1)}\Bigg).\label{thetak}
\end{align}
By the identity  $\Gamma (z)\Gamma \left(z+{\frac{1}{2}}\right)=2^{1-2z}{\sqrt {\pi }}\,\Gamma (2z),$ it follows that
\[-\frac{\pi c_1(\kappa)}{2^{\frac{\kappa+1}{2}+1}\Gamma(\frac{\kappa+1}{2})\sin\left(\frac{\kappa-1}{2} \pi\right)}=\frac{\sqrt{\pi}}{2^{\kappa+1}\Gamma(\frac{\kappa+1}{2})\Gamma(\frac{\kappa}{2})\cos\left(\frac{\kappa}{2} \pi\right)}=\frac{c_2(\kappa)}{2},\]
which completes the proof of (\ref{spectral}).

From $ K_{\nu}\left(|u|\right)\geq 0,$ $u\in[0,\infty),$ and
\[1-\theta_\kappa(|u|)=\frac{c_1(\kappa)}{c_2(\kappa)}K_{\frac{\kappa-1}{2}}\left(|u| \right)|u|^{\frac{1-\kappa}{2}}\]
it follows that $\theta_\kappa(u)\leq 1$ on $u\in[0,\infty).$

By the integral representation {\rm\cite[10.32.8]{NIST:DLMF}}
\[K_{\frac{\kappa-1}{2}}\left(|u| \right)|u|^{\frac{1-\kappa}{2}} =\frac{\sqrt{\pi}}{2^{\frac{\kappa-1}{2}} \Gamma\left(\frac{\kappa}{2}\right)}\int_1^{+\infty}\frac{e^{-|u|t}}{(t^2-1)^{1-\frac{\kappa}{2}}}dt \le \frac{\sqrt{\pi}}{2^{\frac{\kappa-1}{2}} \Gamma\left(\frac{\kappa}{2}\right)}\int_1^{+\infty}\frac{1}{(t^2-1)^{1-\frac{\kappa}{2}}}dt,\]
where the last integral is finite, which proves the boundedness of $\theta_\kappa(u)$ on $u\in[0,\infty).$

By rewriting (\ref{thetak}) one obtains
\begin{align}\theta_\kappa (|u|) &= \Gamma\left(\frac{\kappa+1}{2}\right)\cdot\frac{\left({|u|/2}\right)^{1-\kappa}}{\Gamma\left(\frac{3-\kappa}{2}\right)}  - \frac{|u|^{2}}{2(\kappa+1) } +\Gamma\left(\frac{\kappa+1}{2}\right)\Bigg(\sum_{m=1}^\infty\frac{(|u|/2)^{2m-\kappa+1}}{m!\,\Gamma(m-{\frac{\kappa-1}{2}}+1)}\nonumber\\
& - \sum_{m=2}^\infty\frac{(|u|/2)^{2m}}{m!\,\Gamma(m+{\frac{\kappa-1}{2}}+1)}\Bigg),\nonumber
\end{align}
which gives the asymptotic in (\ref{condition1}).

Using the asymptotic behaviour of the modified Bessel function of the second kind for large values of its argument, see {\rm\cite[10.25.3]{NIST:DLMF},} one obtains
\[
 f_{\kappa,0}(\lambda)  ={c_1(\kappa)}K_{\frac{\kappa-1}{2}}\left(|\lambda| \right)|\lambda|^{\frac{\kappa-1}{2}}\sim \frac{c_{1}(\kappa)\sqrt{\pi}}{\sqrt{2}} |\lambda|^{\frac{\kappa-2}{2}} e^{-|\lambda|},\quad |\lambda| \to \infty,
\]
which gives (\ref{condition2}) and completes the proof.
\end{proof}
\begin{Remark}
Lemma~{\rm\ref{lem1}} show that in this case, we have a generalisation of Condition~D${}^\prime.$ It demonstrates that the spectral density exhibits two singularities at the points $\pm\omega$ with a power-law type of $1-\kappa.$ Additionally, it provides the rate of deviation from this power-law behaviour as $\lambda \to \pm\omega.$
\end{Remark}
Note that the process $\eta(x)$ has the spectral representation
\[\eta(x)=\int_{\mathbb R} e^{i\lambda x}\sqrt{f_{\kappa,\omega}(\lambda)}W(d\lambda),\]
where $W(\cdot)$ is the complex white noise Gaussian random measure, and $f_{\kappa,\omega}(\cdot)$ is given by~{\rm(\ref{spectral})}.
\begin{theorem}\label{the3} Consider the random field $u(t,x),$ $t>0,$ $x\in \mathbb  R$ defined by {\rm (\ref{convolution})} with $\eta(x),$ $ x\in \mathbb  R,$ that has a covariance function satisfying Condition~D${}^{\prime\prime}.$   If $\varepsilon \to 0,$ then the finite-dimensional distributions of the random fields
    \begin{equation*}
        U_{\varepsilon}(t,x)=\varepsilon^{-1/4}\,u\left(\frac{t}{\varepsilon}, \frac{x}{\sqrt{\varepsilon}}\right),
    \end{equation*}
    converge weakly to the finite-dimensional distributions of the zero-mean Gaussian random field
   \begin{equation}\label{*1}
       U_{0}(t,x)=\sqrt{\frac{c_2(\kappa)}{w^{1-\kappa}}(1-\theta_\kappa(|w|))}\int_{\mathbb R}e^{i\lambda x-\mu t\lambda^{2}}W(d\lambda),
   \end{equation}
 with the covariance function
 \begin{align}\label{cov3}
     Cov( U_{0}&(t,x), U_{0}(t',x'))= \mathbb {E} U_{0}(t,x)U_{0}(t',x') =\frac{c_{2}(\kappa)\sqrt{\pi}(1-\theta_\kappa(|w|))}{\sqrt{\mu}w^{1-\kappa}}\cdot\frac{e^{-\frac{(x-x')^2}{4\mu(t+t')}}}{\sqrt{t+t'}}.
 \end{align}
\end{theorem}

\begin{proof}
   It holds in the sense of the finite-dimensional distributions that
\[ U_{\varepsilon}(t,x) = \frac{1}{\varepsilon^{1/4}}\int_{\mathbb{R}} e^{i\lambda \frac{x}{\sqrt{\varepsilon}}-\mu \lambda^{2}{t}/{\varepsilon}}Z(d\lambda) = \frac{1}{\varepsilon^{1/4}}\int_{\mathbb{R}} e^{i\lambda \frac{x}{\sqrt{\varepsilon}}-\mu \lambda^{2}{t}/{\varepsilon}}\sqrt{f(\lambda)}W(d\lambda). \]
By (\ref{spectral}), the change of variables $\lambda = \tilde{\lambda}\sqrt{\varepsilon}$ and using  the Brownian scaling property $W(d(\tilde{\lambda}\varepsilon))\overset{d}{=} \sqrt{\varepsilon} W(d\tilde{\lambda}),$ one obtains
    \begin{align}\label{*}
   U_{\varepsilon}(t,x) &\overset{d}{=}\frac{1}{\varepsilon^{1/4}}\int_{\mathbb{R}} e^{i\tilde{\lambda}x-\mu\tilde{\lambda}^{2}t}\left(f(\tilde{\lambda}\sqrt{\varepsilon})\right)^{1/2}W(d\tilde{\lambda}\sqrt{\varepsilon}) \nonumber \\
    \overset{d}{=}&\int_{\mathbb{R}} e^{i\tilde{\lambda}x-\mu \tilde{\lambda^2}t}\left(f(\tilde{\lambda}\sqrt{\varepsilon})\right)^{1/2}W(d\tilde{\lambda}) \nonumber \\
    \overset{d}{=}&\int_{\mathbb{R}} e^{i\tilde{\lambda}x-\mu \tilde\lambda^{2}t}\sqrt{\frac{c_2(\kappa)}{2}\left(\frac{1-\theta_\kappa(|\tilde{\lambda}\sqrt{\varepsilon}+w|)}{|\tilde{\lambda}\sqrt{\varepsilon}+w|^{1-\kappa}}+\frac{1-\theta_\kappa(|\tilde{\lambda}\sqrt{\varepsilon}-w|)}{|\tilde{\lambda}\sqrt{\varepsilon}-w|^{1-\kappa}}\right)}W(d\tilde{\lambda}).
\end{align}
It follows from (\ref{*1}) and (\ref{*}) that
    \begin{align*}
R(t&,x) = \mathbb {E}\left(U_{\varepsilon}(t,x)-U_{0}(t,x)\right)^{2} \nonumber \\
&= \mathbb {E}\Bigg(\int_{\mathbb{R}} e^{i{\lambda}x-\mu {\lambda}^{2}t}\sqrt{\frac{c_2(\kappa)}{2}\left(\frac{1-\theta_\kappa(|{\lambda}\sqrt{\varepsilon}+w|)}{|{\lambda}\sqrt{\varepsilon}+w|^{1-\kappa}}+\frac{1-\theta_\kappa(|{\lambda}\sqrt{\varepsilon}-w|)}{|{\lambda}\sqrt{\varepsilon}-w|^{1-\kappa}}\right)}W(d{\lambda}) \nonumber \\
& -\sqrt{\frac{c_2(\kappa)}{w^{1-\kappa}}(1-\theta_\kappa(|w|))}\int_{\mathbb{R}}e^{i\lambda x-\mu t\lambda^{2}}W(d{\lambda})\Bigg)^{2} = \frac{c_2(\kappa)}{w^{1-\kappa}}(1-\theta_\kappa(|w|))\nonumber \\
& \times \int_{\mathbb{R}} e^{-2\mu {\lambda}^{2}t} \left(\sqrt{\frac{w^{1-\kappa}}{2-2\theta_\kappa(|w|)}\left(\frac{1-\theta_\kappa(|{\lambda}\sqrt{\varepsilon}+w|)}{|{\lambda}\sqrt{\varepsilon}+w|^{1-\kappa}}+\frac{1-\theta_\kappa(|{\lambda}\sqrt{\varepsilon}-w|)}{|{\lambda}\sqrt{\varepsilon}-w|^{1-\kappa}}\right)} -1\right)^{2} d{\lambda}\nonumber \\
& = \frac{c_2(\kappa)}{w^{1-\kappa}}(1-\theta_\kappa(|w|))\int_{\mathbb{R}} e^{-2\mu{\lambda}^{2}t}\left(\sqrt{Q_{\varepsilon}({\lambda})}-1\right)^{2} d{\lambda},
\end{align*}
where $Q_{\varepsilon}({\lambda}):=\frac{w^{1-\kappa}}{2-2\theta_\kappa(|w|)}\left(\frac{1-\theta_\kappa(|{\lambda}\sqrt{\varepsilon}+w|)}{|{\lambda}\sqrt{\varepsilon}+w|^{1-\kappa}}+\frac{1-\theta_\kappa(|{\lambda}\sqrt{\varepsilon}-w|)}{|{\lambda}\sqrt{\varepsilon}-w|^{1-\kappa}}\right).$

By continuity properties of the function $\theta_\kappa(\cdot),$  for each $\omega>0$ and $\lambda \in \mathbb{R}$ it holds
\[
 \lim_{\varepsilon \to 0}Q_{\varepsilon}({\lambda})=\lim_{\varepsilon \to 0}\frac{w^{1-\kappa}}{2-2\theta_\kappa(|w|)}\left(\frac{1-\theta_\kappa(|{\lambda}\sqrt{\varepsilon}+w|)}{|{\lambda}\sqrt{\varepsilon}+w|^{1-\kappa}}+\frac{1-\theta_\kappa(|{\lambda}\sqrt{\varepsilon}-w|)}{|{\lambda}\sqrt{\varepsilon}-w|^{1-\kappa}}\right)=1.
 \]
 Note that $\sqrt{Q_{\varepsilon}({\lambda})}-1$ converges pointwise to 0 and $(\sqrt{Q_{\varepsilon}({\lambda})}-1)^2\le Q_{\varepsilon}({\lambda})+1.$ Then,  one can apply the generalized Lebesgue's dominated convergence theorem~\mbox{\cite[p.59]{Folland1999}}. To justify its conditions, one needs to show that
\begin{equation}\label{gdt} \lim_{\varepsilon\to 0}\int_{\mathbb{R}} e^{-2\mu{\lambda}^{2}t} (Q_{\varepsilon}({\lambda})+1) d{\lambda} =\int_{\mathbb{R}} e^{-2\mu{\lambda}^{2}t} \lim_{\varepsilon\to 0} (Q_{\varepsilon}({\lambda})+1) d{\lambda}<\infty.
\end{equation}

 The last inequality in (\ref{gdt}) follows from the boundedness of the spectral density at zero
 \[\int_{\mathbb{R}} e^{-2\mu{\lambda}^{2}t} \lim_{\varepsilon\to 0} (Q_{\varepsilon}({\lambda})+1) d{\lambda} = (Q_{0}({0})+1) \int_{\mathbb{R}} e^{-2\mu{\lambda}^{2}t}d{\lambda} <\infty.\]

Now, let us consider the integral on the left-hand side of (\ref{gdt})
\begin{align}\label{int2}
\int_{\mathbb{R}} e^{-2\mu{\lambda}^{2}t} (Q_{\varepsilon}({\lambda})& +1) d{\lambda} = \frac{w^{1-\kappa}}{2-2\theta_\kappa(|w|)}\left(\int_{\mathbb{R}} e^{-2\mu{\lambda}^{2}t}\frac{1-\theta_\kappa(|{\lambda}\sqrt{\varepsilon}+w|)}{|{\lambda}\sqrt{\varepsilon}+w|^{1-\kappa}}d\lambda\nonumber \right.\\
& \left.+\int_{\mathbb{R}} e^{-2\mu{\lambda}^{2}t}\frac{1-\theta_\kappa(|{\lambda}\sqrt{\varepsilon}-w|)}{|{\lambda}\sqrt{\varepsilon}-w|^{1-\kappa}}d\lambda\right)+\int_{\mathbb{R}} e^{-2\mu{\lambda}^{2}t} d{\lambda}\nonumber\\
 &=:\frac{w^{1-\kappa}}{2-2\theta_\kappa(|w|)}(I_1(\varepsilon)+I_2(\varepsilon))+\int_{\mathbb{R}} e^{-2\mu{\lambda}^{2}t}  d{\lambda}.
 \end{align}

Due to the symmetry, we will study only $I_2(\varepsilon).$ Let us introduce the change of variables $u:={\lambda}\sqrt{\varepsilon}$ and split the integration  in $I_2(\varepsilon)$ into two regions, $|u| \leq \varepsilon^{1/2-\delta}$ and $|u| > \varepsilon^{1/2-\delta},$ $ \delta\in (0,1/2):$
\begin{align}\label{**}
    I_2(\varepsilon) = \varepsilon^{-1/2}\int_{|u| \leq \varepsilon^{1/2-\delta}} e^{-2\mu t \frac{u^2}{\varepsilon}} \,\frac{1-\theta_\kappa(|u-\omega|)}{|u-\omega|^{1-\kappa}} du \nonumber\\ + \varepsilon^{-1/2}\int_{|u| > \varepsilon^{1/2-\delta}} e^{-2\mu t \frac{u^2}{\varepsilon}} \,\frac{1-\theta_\kappa(|u-\omega|)}{|u-\omega|^{1-\kappa}} du.
\end{align}
The first integral in (\ref{**}) can be bounded  as
\[\frac{1-\theta_\kappa(\omega)}{\omega^{1-\kappa}}\int_{|u| \leq \varepsilon^{1/2-\delta}} \frac{e^{-2\mu t \frac{u^2}{\varepsilon}}}{\varepsilon^{1/2}} du \le \varepsilon^{-1/2}\int_{|u| \leq \varepsilon^{1/2-\delta}} e^{-2\mu t \frac{u^2}{\varepsilon}} \,\frac{1-\theta_\kappa(|u-\omega|)}{|u-\omega|^{1-\kappa}} du = \]
\begin{equation}\label{bound}
\leq \frac{1-\theta_\kappa(|\varepsilon^{1/2-\delta}-\omega|)}{|\varepsilon^{1/2-\delta}-\omega|^{1-\kappa}}\int_{|u| \leq \varepsilon^{1/2-\delta}} \frac{e^{-2\mu t \frac{u^2}{\varepsilon}}}{\varepsilon^{1/2}} du.
\end{equation}
Noting that
\[\int_{|u| \leq \varepsilon^{1/2-\delta}} \frac{e^{-2\mu t \frac{u^2}{\varepsilon}}}{\varepsilon^{1/2}}du=\int_{|\lambda| \leq \varepsilon^{-\delta}} {e^{-2\mu t {\lambda^2}}}d\lambda,\]
one obtains from (\ref{bound}) that
\begin{equation}\label{Q0}
\lim_{\varepsilon \to 0} \varepsilon^{-1/2}\int_{|u| \leq \varepsilon^{1/2-\delta}} e^{-2\mu t \frac{u^2}{\varepsilon}}
\frac{1-\theta_\kappa(|u-\omega|)}{|u-\omega|^{1-\kappa}} du = Q_{0}({0}) \int_{\mathbb{R}} e^{-2\mu{\lambda}^{2}t}d{\lambda}.
\end{equation}
The second integral in (\ref{**}) can be estimated as
\[0<\varepsilon^{-1/2}\int_{|u| > \varepsilon^{1/2-\delta}} e^{-2\mu t \frac{u^2}{\varepsilon}} \,\frac{1-\theta_\kappa(|u-\omega|)}{|u-\omega|^{1-\kappa}} du\le \frac{e^{-2\mu t \varepsilon^{-2\delta}}}{\varepsilon^{1/2}}\int_{|u| > \varepsilon^{1/2-\delta}} \,\frac{1-\theta_\kappa(|u-\omega|)}{|u-\omega|^{1-\kappa}} du\]
\begin{equation}\label{estint2}
\le \frac{e^{-2\mu t \varepsilon^{-2\delta}}}{\varepsilon^{1/2}}\int_{\mathbb{R}} \,\frac{1-\theta_\kappa(|u-\omega|)}{|u-\omega|^{1-\kappa}} du \to 0,\quad \mbox{when}\quad \varepsilon\to 0,
\end{equation}
as due to the integrability of the spectral density, the integral in (\ref{estint2}) is finite.

Hence, (\ref{gdt}) follows from the results (\ref{int2}), (\ref{**}), (\ref{Q0}) and (\ref{estint2}). Therefore, by the generalized  dominated convergence theorem, $\lim_{\varepsilon \to 0}E{\vert U_{\varepsilon}(t,x)-U_{0}(t,x)\vert}^{2}=0,$ which implies the convergence of finite-dimensional distributions.

Finally, let's proof the equation {(\rm\ref{cov3})}:
\begin{align*}
& Cov(U_{0}(t,x), U_{0}(t',x'))
 =\mathbb {E}\left(\int_{\mathbb R^{2}} \frac{c_{2}(\kappa)}{w^{1-\kappa}}(1-\theta_\kappa(|w|)) e^{i\lambda_{1}x-\mu t\lambda_{1}^{2}} e^{-i\lambda_{2}x'-\mu t'\lambda^2_{2}}W(d\lambda_{1})W(d\lambda_{2})\right)
     \nonumber\\
     &\quad \quad  =\frac{c_{2}(\kappa)}{w^{1-\kappa}}(1-\theta_\kappa(|w|))\int_\mathbb R e^{i\lambda(x-x')}e^{-\mu \lambda^{2}(t+t')} d\lambda =\frac{c_{2}(\kappa)}{w^{1-\kappa}}(1-\theta_\kappa(|w|))\nonumber\\
     &\quad \quad\times\int_{\mathbb R}\cos\left(\lambda(x-x')\right)e^{-\mu(t+t')\lambda^{2}}d\lambda=\frac{c_{2}(\kappa)\sqrt{\pi}(1-\theta_\kappa(|w|))}{\sqrt{\mu}w^{1-\kappa}}\cdot\frac{e^{-\frac{(x-x')^2}{4\mu(t+t')}}}{\sqrt{t+t'}},
\end{align*}
where the last integral was initially computed using Maple software and subsequently verified via \cite[Example 1.16]{henner2009mathematical}.
\end{proof}
\begin{Remark}
It follows from the expression {\rm(\ref{cov3})} for the covariance function that the limit field $U_{0}(t,x)$ is stationary in space but non-stationary in time.
\end{Remark}
 \begin{example}
Figure {\rm\ref{figure2}} shows the covariance function given by {\rm(\ref{cov3})}. To visualise the result as a 3D plot, the variables $\delta=x-x'$ and $t+t'$ were used. The value of $\mu=1$ was selected. Without loss of generality, the other parameters were selected in such a way that the constant multiplier in {\rm(\ref{cov3})} is equal to~1. The covariance takes very large values when differences in spatial coordinates are small and the time is near zero. It sharply decreases to zero as either the temporal interval or the spatial distance between locations increases.
\begin{figure}[tb]
\centering
\includegraphics[width=0.7\linewidth, trim={ 0 0 0 2cm},clip]{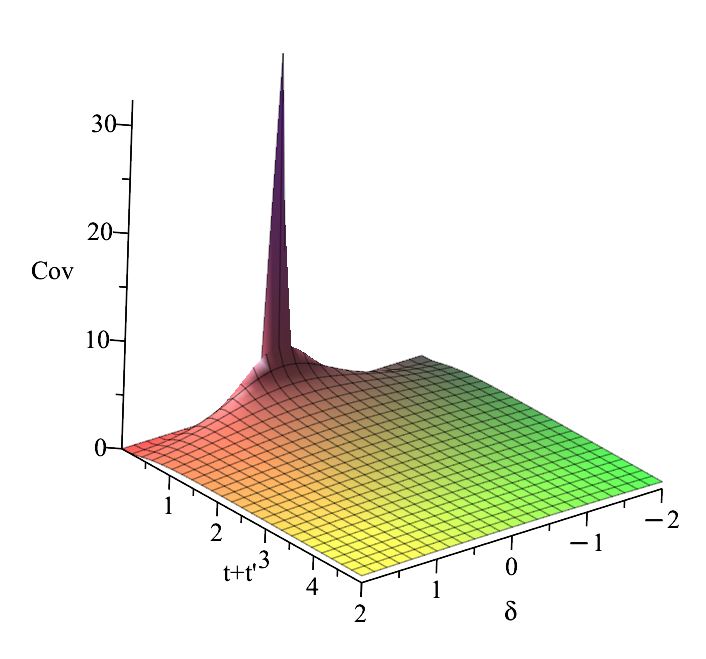}
\caption{ The plot of covariance function {\rm(\ref{cov3})} for $\delta$ and $t+t'.$}
\label{figure2}
\end{figure}
\end{example}

\begin{example}
  Figure {\rm\ref{fig4}} shows two realisations of the random field $U_{0}(t,x).$ To obtain these realizations the following approximation of the stochastic integral in {\rm(\ref{*1})} was used
  \[\int_{\mathbb R} e^{i\lambda x-\mu t \lambda^2} W(d\lambda) \approx \sum _{\lambda_j}e^{i\lambda_{j} x-\mu t \lambda_{j}^2} W(\Delta\lambda_{j}),\]
  where $\Delta\lambda_j= 0.05,$ the sequence of $\lambda_{j}$ was defined as $\{\lambda_j\}=\{-N\times\Delta, -(N-1)\times\Delta,...,0,...,N\times\Delta, (N-1)\times\Delta\}$ and  $N=100.$ Simulation studies confirmed that this width and the number of intervals were sufficient for a close approximation of the field.
  The simulated random fields demonstrated that near $t=0,$ the fields exhibit varying behaviours due to different realizations of random initial conditions. However, as time progresses, the fields decay to zero values.
\begin{figure}[tb]
    \centering
    \subfloat{{\includegraphics[trim=6cm 6cm 5cm 2cm, clip, width=0.48\linewidth]{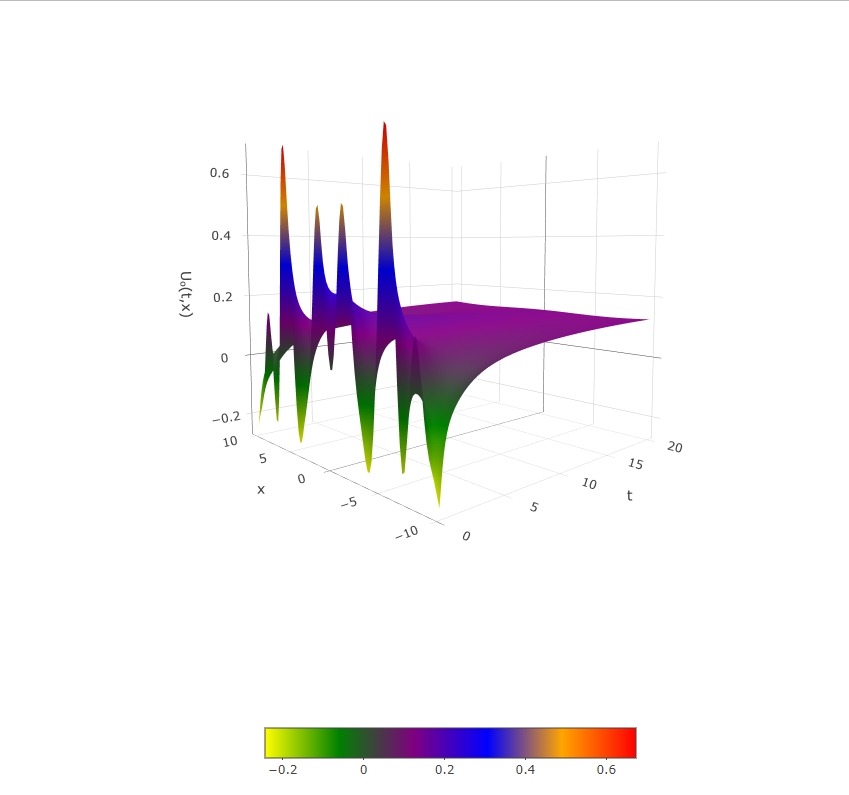}}}
    \subfloat{{\includegraphics[trim=6cm 6cm 5cm 2cm, clip, width=0.48\linewidth]{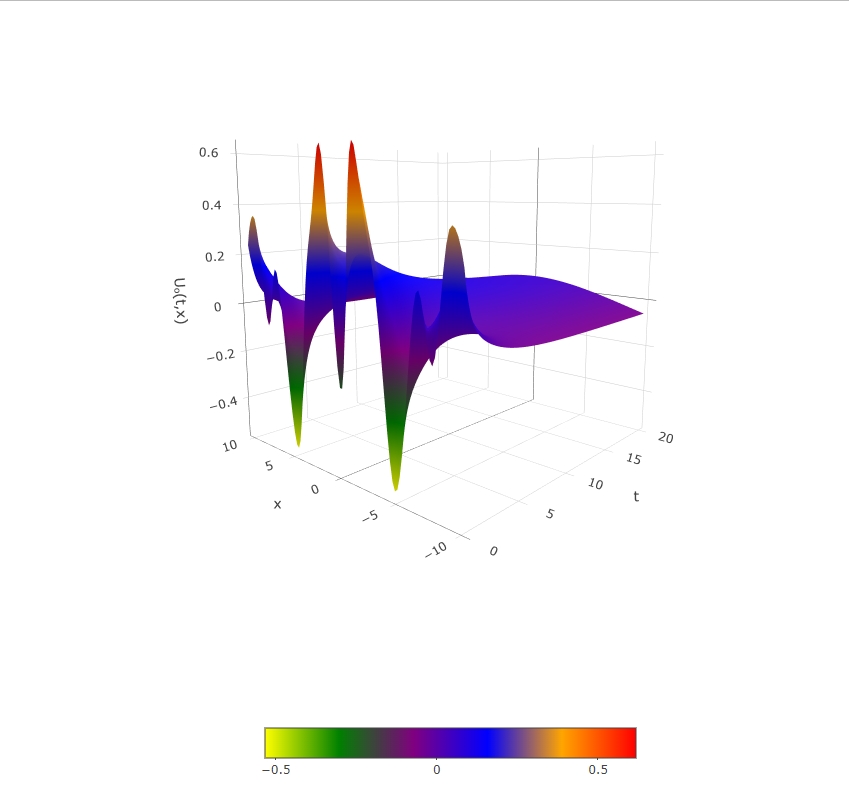}}}
    \caption{Two realizations of $U_{0}(t,x)$.}
    \label{fig4}
\end{figure}
\end{example}

\section{Multiscaling limits for Fractional Riesz-Bessel random PDEs with cyclic long-range dependent initial conditions}\label{sec4}
In this section, we study FRBEs defined by (\ref{Riesz-Bessel equation}) with the initial conditions specified by~(\ref{initial condition}) and~(\ref{covariance of our problem}). In this case, we demonstrate that the analogues of the results in  Section~\ref{sec3} hold true. Similar to Section \ref{sec3}, in this section $m=1$ and $h(x)\equiv x$. For simplicity, we concentrate
on the one-dimensional case with~$d = 1.$


\begin{theorem}\label{the4} Consider the random field $u(t,x),$ $t>0,$ $x\in\mathbb R,$ defined by the  FRBE~{\rm (\ref{Riesz-Bessel equation})} with  $\alpha>1/2$, and the random initial condition $\eta (x),$ $x\in \mathbb R,$ in {\rm(\ref{initial condition})}, that has a covariance function satisfying Condition D${}^{\prime\prime}$. If $\varepsilon\to 0,$ then the finite-dimensional distribution of the random fields
\begin{equation*}
    U_{\varepsilon}(t,x)=\frac{1}{\varepsilon^{\beta/2\alpha}}u\left(\frac{t}{\varepsilon},\frac{x}{\varepsilon^{\beta/\alpha}}\right),
\end{equation*}
converge weakly to the finite-dimensional distribution of the zero mean Gaussian random field
\begin{equation}\label{U_{0}}
    U_{0}(t,x)=\sqrt{\frac{c_{2}(\kappa)}{|w|^{1-\kappa}}(1-\theta_\kappa(|w|))}\int_{\mathbb R}e^{i\lambda x} E_{\beta}\left(-\mu t^{\beta}|\lambda|^{\alpha}\right)W(d\lambda),
\end{equation}
with the covariance function
\begin{align}
  &cov\left(U_{0}(t,x), U_{0}(t',x')\right)\nonumber\\
  &=\frac{2c_{2}(\kappa)(1-\theta_\kappa(|w|))}{|w|^{1-\kappa}}\int_{0}^{
  +\infty}\cos\lambda(x-x')E_{\beta}(-\mu t^{\beta}|\lambda|^{\alpha})E_{\beta}(-\mu t'^{\beta}|\lambda|^{\alpha})d\lambda.\label{cov of Bessel}
\end{align}
\end{theorem}
\begin{proof}
    It follows from (\ref{spectral 1}), that in the sense of the finite-dimensional distributions it holds
    \begin{equation*}
        U_{\varepsilon}(t,x)=\frac{1}{\varepsilon^{\beta/2\alpha}}\int_{\mathbb R}e^{i\lambda \frac{x}{\varepsilon^{\beta/\alpha}}}E_{\beta}\left(-\mu |\lambda|^{\alpha}(1+|\lambda|^{2})^{\frac{\gamma}{2}}\left(\frac{t}{\varepsilon} \right)^{\beta} \right)\sqrt{f(\lambda)}W(d\lambda).
    \end{equation*}
    By (\ref{spectral}), the change of variable $\lambda=\tilde{\lambda}\varepsilon^{\beta/\alpha}$, and using the Brownian scaling property, $W(d(\tilde{\lambda}\varepsilon^{\beta/\alpha}))\overset{d}{=}\varepsilon^{\beta/2\alpha}W(d\tilde{\lambda}),$ one obtains
    \begin{align}\label{U}
        U_{\varepsilon}(t,x)&\overset{d}{=}\frac{1}{\varepsilon^{\beta/2\alpha}}\int_{\mathbb R}e^{i\tilde{\lambda}x} E_{\beta}\left(-\mu |\tilde{\lambda}\varepsilon^{\beta/\alpha}|^{\alpha}(1+|\tilde{\lambda}\varepsilon^{\beta/\alpha}|^{2})^{\gamma/2}\left(\frac{t}{\varepsilon} \right)^{\beta}\right)\sqrt{f(\tilde{\lambda}\varepsilon^{\beta/\alpha})}W(d\tilde{\lambda}\varepsilon^{\beta/\alpha})\nonumber\\
        &\overset{d}{=}\int_{\mathbb R}e^{i\tilde{\lambda}x}E_{\beta}\left(-\mu |\tilde{\lambda}|^{\alpha}(1+|\tilde{\lambda}\varepsilon^{\beta/\alpha}|^{2})^{\gamma/2}t^{\beta}\right)\nonumber\\
        &\times \sqrt{\frac{c_{2}(\kappa)}{2}\left(\frac{1-\theta_\kappa(|\tilde{\lambda}\varepsilon^{\beta/\alpha}+w|)}{|\tilde{\lambda}\varepsilon^{\beta/\alpha}+w|^{1-\kappa}}+\frac{1-\theta_\kappa(|\tilde{\lambda}\varepsilon^{\beta/\alpha}-w|)}{|\tilde{\lambda}\varepsilon^{\beta/\alpha}-w|^{1-\kappa}}  \right)}W(d\tilde{\lambda}).
    \end{align}

Then, by   (\ref{U_{0}}) and (\ref{U})
    \begin{align*}
        R(t,x)&=\mathbb E(U_{\varepsilon}(t,x)-U_{0}(t,x))^{2}\\
        &=\frac{c_{2}(\kappa)}{w^{1-\kappa}}(1-\theta_\kappa(|w|))\int_{\mathbb R}\Biggl(E_{\beta}(-\mu t^{\beta}|\lambda|^{\alpha})\\
        &\times\Biggl(\sqrt{\frac{w^{1-\kappa}}{2(1-\theta_\kappa(|w|))}\left(\frac{1-\theta_\kappa(|\lambda\varepsilon^{\beta/\alpha}+w|)}{|\lambda\varepsilon^{\beta/\alpha}+w|^{1-\kappa}}+\frac{1-\theta_\kappa(|\lambda\varepsilon^{\beta/\alpha}-w|)}{|\lambda\varepsilon^{\beta/\alpha}-w|^{1-\kappa}}  \right)} \\
        &\times \frac{ E_{\beta}\left(-\mu|\lambda|^{\alpha}(1+|\lambda\varepsilon^{\beta/\alpha}|^{2})^{\gamma/2} t^{\beta}\right)}{ E_{\beta}(-\mu t^{\beta}|\lambda|^{\alpha})}-1\Biggl)\Biggl)^{2}d\lambda\\
        &=\frac{c_{2}(\kappa)}{w^{1-\kappa}}(1-\theta_\kappa(|w|))\int_{\mathbb R}\left( E_{\beta}(-\mu t^{\beta}|\lambda|^{\alpha})(\tilde{Q}_{\varepsilon^{\beta/\alpha}}(\lambda)-1)\right)^{2}d\lambda,
    \end{align*}
    where

    $\tilde{Q}_{\varepsilon^{\beta/\alpha}}(\lambda):=\frac{ E_{\beta}\left(-\mu|\lambda |^{\alpha}(1+|\lambda\varepsilon^{\beta/\alpha}|^{2})^{\gamma/2}t^{\beta}\right)}{ E_{\beta}(-\mu t^{\beta}|\lambda|^{\alpha})}\sqrt{\frac{w^{1-\kappa}}{2(1-\theta_\kappa(|w|))}\left(\frac{1-\theta_\kappa(|\lambda\varepsilon^{\beta/\alpha}+w|)}{|\lambda\varepsilon^{\beta/\alpha}+w|^{1-\kappa}}+\frac{1-\theta_\kappa(|\lambda\varepsilon^{\beta/\alpha}-w|)}{|\lambda\varepsilon^{\beta/\alpha}-w|^{1-\kappa}}  \right)}.$

Now, it follows from
$\lim_{\varepsilon\to 0} \tilde{Q}_{\varepsilon^{\beta/\alpha}}(\lambda)=1$ that $\tilde{Q}_{\varepsilon^{\beta/\alpha}}(\lambda)-1$ converges pointwise to zero. As $\left(\tilde{Q}_{\varepsilon^{\beta/\alpha}}(\lambda)-1\right)^{2} \leq \left(\tilde{Q}_{\varepsilon^{\beta/\alpha}}(\lambda)\right)^{2}+1,$  one can apply the generalized Lebesgue's dominated convergence theorem.

To justify its conditions, one needs to show that the following limit exists and finite
\begin{equation}\label{justify condition on Bessel}
    \lim_{\varepsilon\to 0}\int_{\mathbb R} E_{\beta}^{2}(-\mu t^{\beta}|\lambda|^{\alpha})\left((\tilde{Q}_{\varepsilon^{\beta/\alpha}}(\lambda))^{2}+1\right)d\lambda=\int_{\mathbb R} E_{\beta}^{2}(-\mu t^{\beta}|\lambda|^{\alpha})\lim_{\varepsilon\to 0}\left((\tilde{Q}_{\varepsilon^{\beta/\alpha}}(\lambda))^{2}+1\right)d\lambda.
\end{equation}

The boundedness of the last term in (\ref{justify condition on Bessel}) follows from the boundedness of the spectral density at zero
\begin{equation}\label{finite}
    \int_{\mathbb R} E_{\beta}^{2}(-\mu t^{\beta}|\lambda|^{\alpha})\lim_{\varepsilon\to 0}\left((\tilde{Q}_{\varepsilon^{\beta/\alpha}}(\lambda))^{2}+1\right)d\lambda=\left((\tilde{Q}_{0}(0))^{2}+1\right)\int_{\mathbb R} E_{\beta}^{2}(-\mu t^{\beta}|\lambda|^{\alpha})d\lambda<\infty.
\end{equation}
To show that the integral in (\ref{finite}) is finite, let us change the variables as $u=\mu t^{\beta}\lambda^{\alpha}.$ Then $\lambda=\left(u/(\mu t^{\beta})\right)^{\frac{1}{\alpha}},$  $d\lambda=u^{{1/\alpha}-1}/(\alpha(\mu t^{\beta})^{1/\alpha})du,$ and it follows from (\ref{mattig cond}) that for $A>0$ it holds:
\begin{align*}
    \int_{0}^{+\infty}E_{\beta}^{2}(-u)u^{\frac{1}{\alpha}-1}du&=\int_{0}^{A}E_{\beta}^{2}(-u)u^{\frac{1}{\alpha}-1}du+\int_{A}^{+\infty}E_{\beta}^{2}(-u)u^{\frac{1}{\alpha}-1}du\nonumber\\
    &\leq\int_{0}^{A}\frac{u^{{1/\alpha}-1}}{\left(1+\frac{u}{\Gamma(1+\beta)}\right)^{2}}du+\int_{A}^{+\infty}\frac{u^{{1/\alpha}-1}}{\left(1+\frac{u}{\Gamma(1+\beta)}\right)^{2}}du\\
    &\leq\int_{0}^{A}\frac{du}{u^{1-{1/\alpha}}}+\int_{A}^{+\infty}\frac{\Gamma(1+\beta)}{u^{{3}-{1/\alpha}}}\,du.
\end{align*}
The integrals above are finite if $\alpha>1/2.$ It follows from the upper and lower bound in (\ref{mattig cond}) that the condition $\alpha>1/2$ is sufficient and necessary for the boundedness of the integral in (\ref{finite}).

Now, let us consider the integral on the left-hand side of (\ref{justify condition on Bessel})
\begin{align}\label{3I}
     \int_{\mathbb R}& E_{\beta}^{2}(-\mu t^{\beta}|\lambda|^{\alpha})\left((\tilde{Q}_{\varepsilon^{\beta/\alpha}}(\lambda))^{2}+1\right)d\lambda=\int_{\mathbb R} E_{\beta}^{2}(-\mu t^{\beta}|\lambda|^{\alpha})\nonumber\\
     &\times\Biggl(\Biggl(\frac{E_{\beta}\left(-\mu|\lambda|^{\alpha}(1+|\lambda\varepsilon^{\beta/\alpha}|^{2})^{\gamma/2}t^{\beta}\right)}{ E_{\beta}(-\mu t^{\beta}|\lambda|^{\alpha})} \nonumber\\ &\times \sqrt{\frac{w^{1-\kappa}}{2(1-\theta_\kappa(|w|))}\left(\frac{1-\theta_\kappa(|\lambda\varepsilon^{\beta/\alpha}+w|)}{|\lambda\varepsilon^{\beta/\alpha}+w|^{1-\kappa}}+\frac{1-\theta_\kappa(|\lambda\varepsilon^{\beta/\alpha}-w|)}{|\lambda\varepsilon^{\beta/\alpha}-w|^{1-\kappa}}  \right)}\Biggl)^{2}+1\Biggl)d\lambda\nonumber\\&=\frac{w^{1-\kappa}}{2(1-\theta_\kappa(|w|))}\Biggl(\int_{\mathbb R}\left( E_{\beta}(-\mu |\lambda|^{\alpha}(1+|\lambda\varepsilon^{\beta/\alpha}|^{2})^{\gamma/2}t^{\beta}) \right)^{2}\times\frac{1-\theta_\kappa(|\lambda\varepsilon^{\beta/\alpha}+w|)}{|\lambda\varepsilon^{\beta/\alpha}+w|^{1-\kappa}}d\lambda\nonumber\\&+\int_{\mathbb R}\left( E_{\beta}(-\mu |\lambda|^{\alpha}(1+|\lambda\varepsilon^{\beta/\alpha}|^{2})^{\gamma/2}t^{\beta} )\right)^{2}\times\frac{1-\theta_\kappa(|\lambda\varepsilon^{\beta/\alpha}-w|)}{|\lambda\varepsilon^{\beta/\alpha}-w|^{1-\kappa}}d\lambda \Biggl)\nonumber\\
     &+\int_{\mathbb R} E_{\beta}^{2}(-\mu t^{\beta}|\lambda|^{\alpha})d\lambda\nonumber\\&=:\frac{w^{1-\kappa}}{2(1-\theta_\kappa(|w|))}\left(I_{1}(\varepsilon^{\beta/\alpha})+I_{2}(\varepsilon^{\beta/\alpha}) \right)+\int_{\mathbb R}E_{\beta}^{2}(-\mu t^{\beta}|\lambda|^{\alpha})d\lambda.
\end{align}
Due to the symmetry, we will study only $I_{2}(\varepsilon^{\beta/\alpha}).$
Let us split the integration in $I_{2}(\varepsilon^{\beta/\alpha})$  into two regions, $|\lambda|\leq \varepsilon^{-\delta}$ and $|\lambda|> \varepsilon^{-\delta},$ $\delta \in (\beta/(2\alpha^2),\beta/\alpha),$
\begin{align}\label{first integral}
    I_{2}(\varepsilon^{\beta/\alpha})&=\int_{|\lambda|\leq \varepsilon^{-\delta}}\left( E_{\beta}(-\mu |\lambda|^{\alpha}(1+|\lambda\varepsilon^{\beta/\alpha}|^{2})^{\gamma/2}t^{\beta} )\right)^{2}\times\frac{1-\theta_\kappa(|\lambda\varepsilon^{\beta/\alpha}-w|)}{|\lambda\varepsilon^{\beta/\alpha}-w|^{1-\kappa}}d\lambda \nonumber\\& \hspace{-4mm}+ \int_{|\lambda|> \varepsilon^{-\delta}}\left( E_{\beta}(-\mu |\lambda|^{\alpha}(1+|\lambda\varepsilon^{\beta/\alpha}|^{2})^{\gamma/2}t^{\beta} )\right)^{2}\times\frac{1-\theta_\kappa(|\lambda\varepsilon^{\beta/\alpha}-w|)}{|\lambda\varepsilon^{\beta/\alpha}-w|^{1-\kappa}}d\lambda.
\end{align}

By the complete monotonicity property of the Mittag-Leffler functions of negative arguments, see \cite[Proposition 3.10]{gorenflo2020mittag}, we obtain the following lower and upper bounds for the first integral in (\ref{first integral})
\begin{align}
&\frac{1-\theta_\kappa(\omega)}{\omega^{1-\kappa}}
\int_{\mathbb R} \mathbbm{1}_{|\lambda|\leq \varepsilon^{-\delta}}\left( E_{\beta}(-\mu |\lambda|^{\alpha}(1+|\lambda\varepsilon^{\beta/\alpha}|^{2})^{\gamma/2}t^{\beta} )\right)^{2}d\lambda \label{lbound}\\
\le &\int_{|\lambda|\leq \varepsilon^{-\delta}}\left( E_{\beta}(-\mu |\lambda|^{\alpha}(1+|\lambda\varepsilon^{\beta/\alpha}|^{2})^{\gamma/2}t^{\beta} )\right)^{2}\times\frac{1-\theta_\kappa(|\lambda\varepsilon^{\beta/\alpha}-w|)}{|\lambda\varepsilon^{\beta/\alpha}-w|^{1-\kappa}}d\lambda \nonumber\\
& \le \int_{|\lambda|\leq \varepsilon^{-\delta}}\left( E_{\beta}(-\mu |\lambda|^{\alpha}t^{\beta} )\right)^{2}\times\frac{1-\theta_\kappa(|\lambda\varepsilon^{\beta/\alpha}-w|)}{|\lambda\varepsilon^{\beta/\alpha}-w|^{1-\kappa}}d\lambda \nonumber\\
& \le
\frac{1-\theta_\kappa(|\varepsilon^{\beta/\alpha-\delta}-\omega|)}{|\varepsilon^{\beta/\alpha-\delta}-\omega|^{1-\kappa}}\int_{|\lambda|\leq \varepsilon^{-\delta}} E_{\beta}^{2}(-\mu |\lambda|^{\alpha}t^{\beta} )d\lambda\nonumber,
\end{align}
where $\mathbbm{1}_{A}(\cdot)$ is the indicator function of a set $A.$

Then,
\begin{equation}\label{1I}
    \lim_{\varepsilon\to 0}\frac{1-\theta_\kappa(|\varepsilon^{\beta/\alpha-\delta}-\omega|)}{|\varepsilon^{\beta/\alpha-\delta}-\omega|^{1-\kappa}}\int_{|\lambda|\leq \varepsilon^{-\delta}} E_{\beta}^{2}(-\mu |\lambda|^{\alpha}t^{\beta} )d\lambda=\frac{1-\theta_\kappa(|w|)}{|w|^{1-\kappa}}\int_{\mathbb R} E_{\beta}^{2}(-\mu |\lambda|^{\alpha}t^{\beta} )d\lambda,
\end{equation}
where, as was proved for (\ref{finite}),  the last integral is finite.

From that, the integrals in (\ref{lbound}) are uniformly bounded. Noting that the integrands in (\ref{lbound}) are bounded by $E_{\beta}^{2}(-\mu |\lambda|^{\alpha}t^{\beta} )$ and pointwise converge to this bound, by the dominated convergence theorem,  one obtains that the lower bound converges to the same values as the upper one.
Therefore, the first integral in (\ref{first integral}) converges to $(1-\theta_\kappa(|w|))|w|^{\kappa-1}\int_{\mathbb R} E_{\beta}^{2}(-\mu |\lambda|^{\alpha}t^{\beta} )d\lambda.$

The second integral in (\ref{first integral}) can be estimated as
\begin{align}\label{second integral}
    0&<\int_{|\lambda|> \varepsilon^{-\delta}}\left( E_{\beta}(-\mu |\lambda|^{\alpha}(1+|\lambda\varepsilon^{\beta/\alpha}|^{2})^{\gamma/2}t^{\beta} )\right)^{2}\times\frac{1-\theta_\kappa(|\lambda\varepsilon^{\beta/\alpha}-w|)}{|\lambda\varepsilon^{\beta/\alpha}-w|^{1-\kappa}}d\lambda\nonumber\\&\leq \frac{ E_{\beta}^{2}(-\mu \varepsilon^{-\delta\alpha}(1+|\varepsilon^{{\beta/\alpha}-\delta}|^{2})^{\gamma/2}{t}^{\beta}) }{ \varepsilon^{\beta/\alpha}}\int_{|u|>\varepsilon^{{\beta/\alpha}-\delta}}\frac{1-\theta_\kappa(|u-w|)}{|u-w|^{1-\kappa}}du\nonumber\\&\leq \frac{E_{\beta}^{2}(-\mu {t}^{\beta}\varepsilon^{-\delta\alpha}) }{ \varepsilon^{\beta/\alpha}}\int_{\mathbb R}\frac{1-\theta_\kappa(|u-w|)}{|u-w|^{1-\kappa}}du\to 0, \quad \mbox{when}\ \varepsilon\to 0,
    \end{align}
     as by the upper bound in (\ref{mattig cond}) and the condition $\delta \in (\beta/(2\alpha^2),\beta/\alpha)$
 \[ \frac{E_{\beta}^{2}(-\mu {t}^{\beta}\varepsilon^{-\delta\alpha}) }{ \varepsilon^{\beta/\alpha}}\le \frac{1}{ \varepsilon^{\beta/\alpha}+\mu^{2} {t}^{2\beta}\varepsilon^{\beta/\alpha-2\delta\alpha}/\Gamma^{2}(1+\beta) }\to 0, \quad \mbox{when}\ \varepsilon\to 0,  \]
and due to the integrability of the spectral density, the integral in (\ref{second integral}) is finite.

 Hence, the equality in (\rm\ref{justify condition on Bessel}) follows from the results (\rm\ref{3I}), (\rm\ref{first integral}),  (\rm\ref{1I}) and (\ref{second integral}). Thus, by the generalized dominated converges theorem $\lim_{\varepsilon \to 0}E{\vert U_{\varepsilon}(t,x)-U_{0}(t,x)\vert}^{2}=0,$ which implies the convergence of finite dimensional distributions.

 Finally, as the random field $U_{0}(t,x)$ is zero mean, then
 \begin{align}
     &cov(U_{0}(t,x),U_{0}(t',x'))=\mathbb E U_{0}(t,x)U_{\varepsilon}(t^{\prime},x^{\prime})\nonumber\\&=\mathbb E\left(\frac{c_{2}(\kappa)}{|w|^{1-\kappa}}(1-\theta_\kappa(|w|))\int_{\mathbb R^2}  e^{i\lambda_{1}x}E_{\beta}\left(-\mu t^{\beta}|\lambda_{1}|^{\alpha}\right) e^{i\lambda_{2}x^{\prime}} E_{\beta}\left(-\mu t^{\prime\beta}|\lambda_{2}|^{\alpha}\right)W(d\lambda_{1})W(d\lambda_{2})\right)\nonumber\\&=\frac{c_{2}(\kappa)}{|w|^{1-\kappa}}(1-\theta_\kappa(|w|))\int_{\mathbb {R}}e^{i\lambda(x-x^{\prime})}E_{\beta}\left(-\mu t^{\beta}|\lambda|^{\alpha}\right)E_{\beta}\left(-\mu t^{\prime\beta}|\lambda|^{\alpha}\right)d\lambda.\nonumber\\
     &=\frac{2c_{2}(\kappa)(1-\theta_\kappa(|w|))}{|w|^{1-\kappa}}\int_{0}^{
  +\infty}\cos\lambda(x-x')E_{\beta}(-\mu t^{\beta}|\lambda|^{\alpha})E_{\beta}(-\mu t'^{\beta}|\lambda|^{\alpha})d\lambda\nonumber,
 \end{align}
 which completes the proof.
\end{proof}

\begin{example} This example demonstrates how the covariance function of the limit field $U_{0}(t,x)$ given in Theorem~{\rm\ref{the4}}, depends on time, space, and its parameters. Without loss of generality, we assume that in (\ref{cov of Bessel}) the constant $\frac{2c_{2}(\kappa)(1-\theta_\kappa(|w|))}{|w|^{1-\kappa}}=1$. Since the covariance $cov(U_{0}(t,x),U_{0}(t',x'))$ given by (\ref{cov of Bessel}) depends on four variables, $t$, $t'$, $x$, $x'$, along with its parameters, such as $\alpha$, $\beta$,  and $\mu$, we can provide only 3D plots that illustrate the behaviour with respect to a subset of these variables keeping the remaining variables fixed. Given that the field $U_{0}(t,x)$ is stationary in space, we will use the notation $\delta=x-x'$.

Consider the case of $\beta=1/2$, for which ${\displaystyle E_{\frac {1}{2}}(z)=\exp(z^{2})\operatorname {erfc}(-z)}$, where the Gauss error function, denoted by $\operatorname{erf}$, is defined as ${\displaystyle \operatorname {erf} (z)={\frac {2}{\sqrt {\pi }}}\int_{0}^{z}e^{-t^{2}}\mathrm {d}t.}$

For $\mu = t = t' = 1$, the left subplot of Figure~{\rm\ref{fig5}} presents the covariance function $cov(U_{0}(1,x),U_{0}(1,x+\delta))$ for $\delta$ in the range $[0,5]$ and $\alpha$ in the range $[1,2]$. The plot demonstrates that the covariance decreases as the distance $\delta$ between spatial locations increases. Regarding $\alpha$, the covariance takes large values at larger values of $\alpha$ if $\delta$ is close to $0$. For large values of $\delta$, and vice versa, the covariance is decreasing with increasing $\alpha$. Also, the first plot demonstrates that the rate of decrease of the covariance function with respect to $\delta$ is large for large values of $\alpha$.
\begin{figure}[htb]%
    \centering
    \subfloat{{\includegraphics[scale=0.3]{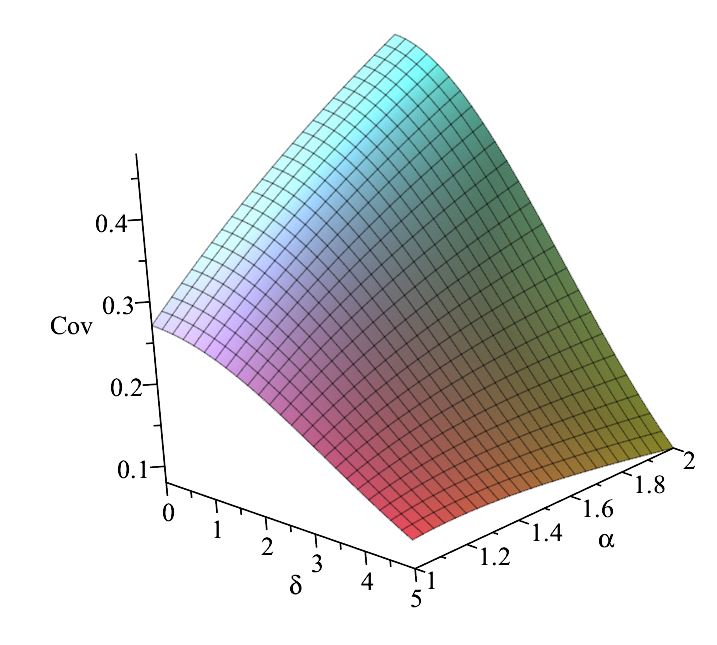} }}%
    \hspace{10mm}
    \subfloat{{\includegraphics[scale=0.4]{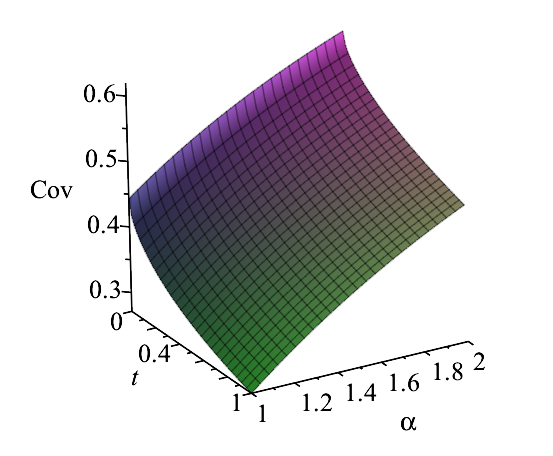} }}%
    \vspace{5mm}
    \caption{Dependence of the covariance function of $U_{0}(t,x)$ on  $\delta,$ $t-t',$  and $\alpha.$}
    \label{fig5}
\end{figure}

For $\delta=0$, i.e., when $x = x'$, the right subplot of Figure~{\rm\ref{fig5}} shows the covariance function $cov(U_{0}(1,x),U_{0}(t,x))$ for the parameter $\alpha$ in the range $[1,2]$, with the fixed value of $t' = 1$ and $t$ in the range $[1,2]$. As expected, the plot demonstrates that the covariance decreases as the time difference $t - t'=t-1$ increases. The covariance is an increasing function of $\alpha.$ However, unlike in the previous subplot, the covariance always increases with $\alpha$. In general, one can conclude that for small time or space distances $\alpha$ plays an important role and in such cases the covariance of $U_{0}(x,t)$ increases with $\alpha$. When time and space distances increase the covariance approaches to $0$, and therefore, the impact of $\alpha$ diminishes.
\end{example}

\section{Conclusion}\label{sec6}
The paper proved multiscaling limit theorems for renormalized solutions of stochastic partial differential equations. Specifically, it examined the cases where initial conditions are subordinated to random processes with cyclic long-range dependence. The spectral and covariance representations for the corresponding limit random fields were derived. Numerical examples were provided to illustrate the obtained theoretical results.

Beyond the framework considered in the cited papers, there are many interesting open problems about properties of the obtained solutions. Investigating the density of the solutions, certain geometric functionals (such as level sets), extensions to the vector case, and the corresponding functionals, as discussed in \cite{Olivera, Omari, Armentano}, are particularly interesting and challenging problems. Studying regularity properties of the asympototic field is an open problem, especially interesting in temporal settings, see, for example, \cite{guo2024samplepathpropertiessmall} and \cite{qian2025temporalregularitystochasticheat}.

Another important future direction is to consider the discretized version and related problems (see Section 4.1 in \cite{LW}). In these cases, the spectral functions involve elliptic Jacobi theta functions, requiring an extension of the proposed methodology to accommodate them.

We complete the paper by discussing new challenges for subordinated cases with Hermite ranks greater than 1. In these cases, the analogous of the known results and the approach developed in the paper are not valid and a new methodology is required.

Specifically, let us consider the initial value problem (\ref{classical heat equation}) and (\ref{initial condition for heat equation}), but with the random initial condition $    u(0,x)=H_{m}(\xi(x)),$ $x\in \mathbb R,$
where $H_{m}(\cdot),$  $m>1,$ is the $m${th} Hermite polynomial.
Note, that in this case $h(x)\equiv H_{m}(x)$ and $\eta(x)=H_{m}(\xi(x))$ in Conditions~A and~C and the Hermite rank is $m>1$ in Condition~B.

Then, $\eta(x)$ can be represented as
\[\eta(x)=\int_{\mathbb R^m}' e^{i(\lambda_{1}+...+\lambda_{m})x}\left(\prod_{j=1}^{m}{f^{1/2}(\lambda_{j})}\right)W(d\lambda_{1})...W(d\lambda_{m}).\]

When $\varepsilon \to 0,$ then, similar to Theorems~\ref{th1}-~\ref{the4},  one can expect that for some constant $\zeta$ there exists a limit of the random fields
$        U_{m,\varepsilon}(t,x)=\varepsilon^{\zeta}\,u_m\left(\frac{t}{\varepsilon}, \frac{x}{\sqrt{\varepsilon}}\right),$
and,  if it is possible to use the dominated convergence theorem or its analogous, then that limit has the form
  \begin{equation}\label{newlim}
   C\int_{\mathbb R^m}e^{i(\lambda_{1}+...+\lambda_{m})x  -\mu t|\lambda_{1}+...+\lambda_{m} | ^{2}}W(d\lambda_{1})...W(d\lambda_{m}),
   \end{equation}
   or
   \[
   C\int_{\mathbb R^m}e^{i(\lambda_{1}+...+\lambda_{m})x} E_{\beta}\left(-\mu t^{\beta}|\lambda_{1}+...+\lambda_{m} |^{\alpha}\right)W(d\lambda_{1})...W(d\lambda_{m}).
\]
However, the above integrals are not correctly defined as, for example, the variance of the first field~(\ref{newlim}),
  \begin{equation}\label{newlim1}
   C\int_{\mathbb R^m}e^{-2\mu t|\lambda_{1}+...+\lambda_{m} | ^{2}}d\lambda_{1}...d\lambda_{m},
   \end{equation}
is infinite. It is easy to see because the integrand takes values close to 1 in the region of infinite volume around hyper-diagonals.

This phenomenon does not occur in the classical long memory case addressed in Theorems~\ref{th1} and \ref{th2}, where the integrands are divided by powers of the norms of individual $\lambda_i,$  $i=1,\ldots,m,$ which makes them integrable. Thus, an interesting question arises as to whether it is possible to derive limit theorems for $m>1,$ and if so, what the appropriate normalizations and limits would be if they exist.

\section*{Acknowledgements} This research was supported under the Australian Research Council's Discovery Projects funding scheme (project number  DP220101680). A.~Olenko would like to thank the Department of Statistics and Probability, Michigan State University, USA, for hosting his sabbatical in 2024 and particularly to Prof.~Y.~Xiao for his support and for providing stimulating discussions, which contributed to refining certain aspects of this paper. He is also grateful to Prof.~E.~Orsingher (Sapienza Università di Roma) for discussions on the general theory of Airy processes. A.~Olenko was partially supported by La Trobe University's SCEMS CaRE and Beyond grant. N. Leonenko would like to thank for support and hospitality during the programmes “Fractional Differential Equations” (2022), “Uncertainly Quantification and Modelling of Materials”(2024) and "Stochastic Systems for Anomalous Diffusion"(2024)  in Isaac Newton Institute for Mathematical Sciences, Cambridge. Also, he was partially supported under the  LMS grant 42997 (UK), grant FAPESP 22/09201-8 (Brazil) and the Croatian Science Foundation (HRZZ) grant Scaling in Stochastic Models (IP-2022-10-8081). The authors would like to thank Prof.~P.~Broadbridge for discussions about various applications of PDEs in physics and cosmology studies.


\bibliographystyle{imsart-nameyear} 
\bibliography{bibliography}       

\end{document}